\documentclass[a4paper,oneside]{amsart}
\usepackage[utf8]{inputenc}

\usepackage{amssymb}
\usepackage{amsthm}
\usepackage{amsmath,amscd}
\usepackage[mathscr]{euscript}
\usepackage[all]{xy}
\usepackage{lmodern}
\usepackage[T1]{fontenc}
\usepackage[textwidth=14cm,hcentering]{geometry}
\usepackage[colorlinks=true,linkcolor=red,citecolor=blue]{hyperref}
\usepackage{mathtools}
\usepackage{tikz}
\usetikzlibrary{cd}
\usetikzlibrary{arrows,positioning}

\title{The additivity of traces in stable $\infty$-categories}
\author{Maxime Ramzi }
\email{maxime.ramzi@math.ku.dk}
\thanks{I was supported financially by École Normale Supérieure under the status of ``fonctionnaire stagiaire'', and was hosted by the GeoTop center at K\o benhavns Universitet}
\newtheorem{thm}{Theorem}[section]
\newtheorem{lm}[thm]{Lemma}
\newtheorem{prop}[thm]{Proposition}
\newtheorem{cor}[thm]{Corollary}

\newtheorem*{thm*}{Theorem}

\theoremstyle{definition}
\newtheorem{defn}[thm]{Definition}
\newtheorem{cons}[thm]{Construction}

\newtheorem{nota}[thm]{Notation}

\newtheorem{rmk}[thm]{Remark}

\newtheorem{obs}[thm]{Observation}

\newcommand{\op}{^{\mathrm{op}}}
\newcommand{\cat}{\mathbf}
\newcommand{\on}{\operatorname}
\newcommand{\id}{\mathrm{id}}
\newcommand{\Fun}{\on{Fun}}
\newcommand{\map}{\on{map}}
\newcommand{\Map}{\on{Map}}

\newcommand{\Sph}{\mathbb S}

\newcommand{\Sp}{\cat{Sp}}

\newcommand{\PrL}{\cat{Pr}^L}

\newcommand{\CAlg}{\mathrm{CAlg}}

\newcommand{\HH}{\mathrm{HH}}
\newcommand{\THH}{\mathrm{THH}}
\newcommand{\Perf}{\mathbf{Perf}}

\newcommand{\Ind}{\mathrm{Ind}}
\newcommand{\Endo}{\mathrm{End}}

\newcommand{\C}{\cat C}
\newcommand{\un}{\mathbf 1}
\newcommand{\tr}{\mathrm{tr}}
\begin{document}
\begin{abstract}
    We prove a version of J.P. May's theorem on the additivity of traces, in symmetric monoidal stable $\infty$-categories. Our proof proceeds via a categorification, namely we use the additivity of topological Hochschild homology as an invariant of stable $\infty$-categories and construct a morphism of spectra $\THH(\C)\to \Endo(\un_\C)$ for $\C$ a stably symmetric monoidal rigid $\infty$-category. We also explain how to get a more general statement involving traces of finite (homotopy) colimits. 
\end{abstract}
\maketitle

\section*{Introduction}
Consider an endomorphism $f$ of a finite dimensional vector space $V$. 

One way to compute the trace of $f$ is to find a sub-vector space stable under $f$, and use the following fact: 
\begin{lm}
Suppose we are given a short exact sequence of endomorphisms of finite dimensional vector spaces:  $$\xymatrix{0\ar[r] \ar[d] & W\ar[r] \ar[d]^f & V \ar[r]\ar[d]^g & Q\ar[d]^h \ar[r] & 0\ar[d]\\ 0\ar[r]&
W\ar[r] & V\ar[r] & Q\ar[r] & 0}$$ 
Then $\tr(g) = \tr(f)+\tr(h)$. 
\end{lm}
The proof is quite simple, as given an appropriate basis for $V$, the matrix of $g$ looks like a block matrix $\begin{pmatrix}M & * \\0 & N\end{pmatrix}$ where $M$ is a matrix for $f$ in $W$, and $N$ a matrix for $h$ in $Q$. 

This proof, however, does not extend to other contexts where traces are nonetheless defined, e.g. over finite spectra (where traces are related to Lefschetz numbers and Euler characteristics, see \cite{DoldPup}). 

In \cite{MayAdd}, J.P. May gives a proof of an analogous statement in a suitable tensor triangulated category (such as the finite stable homotopy category). The proof relies on several diagram chases and has a somewhat technical flavour. 
The same is true for another proof given by M. Groth, K. Ponto and M. Shulman in \cite{GPS} in the setting of derivators. 

The purpose of this short note is to provide a different proof of the analogous statement for stable symmetric monoidal $\infty$-categories based on a categorification of the problem. Our main theorem is : 
 
\begin{thm*}
Let $\C$ be a stably symmetric monoidal $\infty$-category with unit $\un$, and let 
$$\xymatrix{X\ar[r] \ar[d]^f & Y \ar[r]\ar[d]^g & Z\ar[d]^h \\ 
X\ar[r] & Y\ar[r] & Z}$$ 
denote a co/fiber sequence of endomorphisms of dualizable objects.

Then, in $\pi_0\Endo(\un)$, $\tr(g) = \tr(f)+\tr(h)$. 
\end{thm*}

A similar proof appeared for the special case of Euler characteristics in \cite{HSS}. In our presentation we replace algebraic $K$-theory with topological Hochschild homology ($\THH$) to obtain the more general statement.

We also explain how to get a more general statement involving finite (homotopy) colimits, based on \cite{berman}.

\subsection{Outline of the paper}
In section \ref{section : strat}, we explain our general proof strategy, which reduces to constructing a morphism $\THH(\C)\to \Endo(\un)$ with nice properties. 

In section \ref{section : map}, we explain how to construct the morphism in question and prove that it satisfies the desired properties.

In section \ref{section : rig}, we mention a slight generalization of the construction from section \ref{section : map}, which is not necessary for our main theorem but might be of interest for future applications. 

Finally, section \ref{section : mobius} is where we state and prove a slightly more general version of the additivity theorem, for finite homotopy colimits. 
\section*{Conventions}
We use the framework of $\infty$-categories as developped in \cite{HTT} and \cite{HA}. We use the word ``category'' for what is called $\infty$-category in these books, and all our constructions such as co/limits should be interpreted in $\infty$-categories.

We use the expression stably symmetric monoidal category to mean a symmetric monoidal category whose underlying category is stable, and such that the tensor product is exact in each variable. 

$\PrL$ is the category of presentable categories, and $\PrL_{st}$ of presentable stable categories. 

We use $\otimes$ to denote several different things: the smash product of spectra, the Lurie tensor product on $\PrL$, the category of presentable categories and left adjoint functors, the one on stable categories and exact functors, the canonical $\Sp$-tensor on a stable presentable category and the tensor product on our stable category of interest. Which one of these we mean should be clear from context. 

We use without comment the fact that a stable category is canonically enriched in spectra, and let $\Map$ denote the mapping spectrum functor and we use $\map$ for mapping spaces (with similar conventions for $\Endo$ and $\mathrm{end}$).

$\Fun^{ex}$ denotes the category of exact functors, $\Fun^L$ the category of left adjoint functors; and finally $\Sph$ denotes the sphere spectrum.
\section*{Acknowledgements}
I would like to thank Marc Hoyois for helpful discussions relating to this note, and for suggesting that I add the material of section 3. 

I am also grateful to Guillaume Laplante-Anfossi for helpful comments on a first draft.
\section{Strategy of proof of the main theorem}\label{section : strat}
\hspace{\parindent}Suppose $\C$ is a small stably symmetric monoidal category with all objects dualizable. Let $\un$ denote its unit. We will construct in section \ref{section : map} a morphism from the topological Hochschild homology of $\C$ to the endomorphism spectrum of its unit, $\THH(\C)\to \Endo(\un)$ with the property that the composite $\pi_0\map(x,x)\cong \pi_0\Map(x,x)\to \pi_0\THH(\C)\to \pi_0\Endo(\un)$ is the function that sends an endomorphism of $x$ to its trace, for any $x\in\C$.

\begin{rmk}\label{rmk : alldual}
The assumption that all objects of $\C$ are dualizable is not essential to the additivity theorem, but it is essential to the construction of the morphism $$\THH(\cat C)\to \Endo(\un)$$ 
Indeed, because the additivity theorem only concerns dualizable objects, one can safely replace $\C$ by its full subcategory of dualizable objects, which is a stable subcategory closed under tensor products. 
\end{rmk}
We now explain how to deduce the additivity theorem from the existence of such a morphism.

Topological Hochschild homology is an additive invariant in the sense of \cite{BGT}, in particular for exact functors $F,G,H : \cat D\to\C$ lying in a co/fiber sequence $F\to G\to H$, the induced maps satisfy $\THH(G)\simeq \THH(F)+\THH(H)$ as maps $\THH(\cat D)\to \THH(\C)$. \begin{rmk}
There is a proof in \cite{BM} that $\THH$ is additive, but it is also the content of \cite[Section 3]{HSS}, where a relatively elementary proof appears. By \cite[§4.5]{HSS}, the definitions of $\THH$ given in these two papers agree.
\end{rmk}
Recall further that evaluation along the universal endomorphism $\Sph[t]\to \Sph[t]$ induces an equivalence $$\Fun^{ex}(\Perf(\Sph[t]),\C)\to \Fun(\Delta^1/\partial \Delta^1, \C)$$\label{eq : repend}
see e.g. \cite[Proposition 3.8]{KEnd}.

In particular, a co/fiber sequence of endomorphisms $$\xymatrix{X\ar[r] \ar[d]^f & Y \ar[r]\ar[d]^g & Z\ar[d]^h \\ 
X\ar[r] & Y\ar[r] & Z}$$ induces a co/fiber sequence of exact functors $\Perf(\Sph[t])\to \C$ representing these endomorphisms (call them $F,G,H$ respectively). 

Additivity of $\THH$ implies that $\THH(G)\simeq\THH(F)+\THH(H)$, and we have the following commutative diagram, for any exact functor $R : \Perf(\Sph[t])\to\C$: $$\xymatrix{\map(\Sph[t],\Sph[t]) \ar[r] \ar[d] & \map(R(\Sph[t]),R(\Sph[t])) \ar[d] \\ \THH(\Sph[t]) \ar[r] & \THH(\C)}$$

Denoting by $\tilde{f}$ the image in $\pi_0\THH(\C)$ of $f\in \pi_0\map(X,X)$ it follows that $\tilde{g} = \tilde{f}+\tilde{h}$. 

Composing with the morphism $\pi_0\THH(\C)\to\pi_0\Endo(\un)$, using the fact that the composition maps $f$ to its trace, we find that $\tr(g) = \tr(f)+\tr(h)$, which is the additivity theorem, in the following form (see remark \ref{rmk : alldual}): 
\begin{thm}\label{thm : truemain}
Let $\C$ be a stably symmetric monoidal category with unit $\un$, and let 
$$\xymatrix{X\ar[r] \ar[d]^f & Y \ar[r]\ar[d]^g & Z\ar[d]^h \\ 
X\ar[r] & Y\ar[r] & Z}$$ 
denote a co/fiber sequence of endomorphisms of dualizable objects.

Then, in $\pi_0\Endo(\un)$, $\tr(g) = \tr(f)+\tr(h)$. 
\end{thm}
We will write down the complete proof at the end of section \ref{section : map}. 
\begin{rmk}
The notation above is a bit abusive : a cofiber sequence in a stable category is not only a sequence $X\to Y\to Z$, but it is the data of a commutative square $$\xymatrix{X\ar[r]\ar[d]& Y \ar[d] \\0\ar[r] & Z}$$
In particular it comes equipped with a specified nullhomotopy of the composite $X\to Z$. A co/fiber sequence of endomorphisms (or an endomorphism of a co/fiber sequence) has to take this into account, so that the non-abusive way to write such a diagram would involve a whole cube rather than the ladder we drew. 

If $f,g$ are given, the condition for $h:Z\to Z$ to lie in such a co/fiber sequence is not only that the right hand square commute. This explains why in \cite{MayAdd}, the statement looks like ``for any $f,g$, there exists an $h$ such that''.

Up to homotopy, $h$ depends on the specific chosen homotopy witnessing the commutativity of the left hand diagram - any choice of such a homotopy yields possibly different $h$'s as their cofiber. For instance, some choices of the left hand square have some $h$'s nullhomotopic, and some $h$'s \emph{not} nullhomotopic  

The statement in \cite{GPS} is essentially the same as ours, in the setting of stable monoidal derivators, though their proof is rather different. 
\end{rmk}
\begin{rmk}
See \cite{counterex} for a counter-example to the naive statement that an endomorphism of the diagram $X\to Y\to Z$ has the same additivity property. The theorem really is about endomorphisms of co/fiber sequences. 

In a letter to Thomason, Grothendieck explains how this failure of additivity of traces in the bare derived $1$-category was one of his motivations for the development of the formalism of derivators, where mapping cones would be functorial. Stable $\infty$-categories are another formalism in which there is such a functoriality, and this is why they can also be a home for this theorem. 
\end{rmk}
Using the special case of a co/fiber sequence of identity maps, and letting $\chi$ denote the Euler characteristic (i.e. the trace of the identity map), we deduce: 
\begin{cor}
With the same assumptions as in the previous theorem, if $X,Y,Z$ are dualizable and there is a cofiber sequence $X\to Y\to Z$, then $\chi(Y) = \chi(X)+\chi(Z)$. 
\end{cor}
\begin{rmk}
As pointed out in \cite[Remark 6.6]{HSS}, this corollary can be stated in an even stronger way, as the existence of an $S^1$-equivariant trace map $\mathbb K(\C)\to \THH(\C)$, where $\mathbb K$ denotes nonconnective algebraic $K$-theory. 
\end{rmk}
\section{Construction of the morphism}\label{section : map}
In this section we explain how to define the map $\THH(\C)\to \Endo(\un)$, in other words we prove:
\begin{thm}\label{thm : main}
Suppose $\C$ is a small stably symmetric monoidal category where all objects are dualizable, with unit $\un$. Then there exists a map of spectra $$\THH(\C)\to \Endo(\un)$$ such that the composite with the canonical map $\map(x,x)\to \THH(\C)\to \Endo(\un)$ realizes the trace on $\pi_0$, for all $x\in\C$.
\end{thm}
We will explain during the course of the construction what ``the canonical map'' means. For the proof, all that matters is that this map is natural in $(\C,x)$ and that the composite acts as desired on $\pi_0$. From there, the additivity theorem follows as previously explained. 
\begin{rmk}
One can see this map $\THH(\C)\to \Endo(\un)$ as a homotopy-coherent version of the trace - the homotopy coherence taking into account both the action of the automorphism group of $x$, and the cyclic invariance of the trace. 
\end{rmk}
\begin{rmk}\label{rmk : agree}
Suppose $\C = \Perf(R)$ for some commutative ring spectrum $R$. Then $\THH(\Perf(R))\simeq \THH(R) \simeq R\otimes_{R\otimes R}R$, and there is an obvious morphism of commutative ring spectra $R\otimes_{R\otimes R}R\to R$ given by the fact that the relative tensor product of commutative ring spectra is their pushout. We will not need it here, but these two maps agree (see remark \ref{rmk : agreemaps}).
\end{rmk}
The rest of this section is devoted to the proof of theorem \ref{thm : main}. 

The first step is to recall the point of view on $\THH$ that will be most helpful for us (see \cite[§4.5]{HSS}): 
\begin{defn}
Given a small stable category $\C$ and $T\in \Fun^L(\Ind(\C),\Ind(\C))$, $\THH(\C,T)$ denotes the trace of $T$ with respect to the Lurie monoidal structure on $\PrL_{st}$, where $\Ind(\C)$ is dualizable because it is compactly generated \cite[§ 3.3]{BGT}. 

Further, we let $\THH(\C) :=\THH(\C;\id_\C)$
\end{defn}
 Having a look at the proof from the first section, we note that we don't specifically need the map $\map(x,x)\to \THH(\C)$ to be the one induced by the inclusion in the cyclic bar construction, we really only need it to be natural in $(\C,x)$ - although the map we construct is most likely equivalent to that inclusion.

\begin{nota}
Suppose $\cat E,\cat D$ are presentable stable categories, $e\in\cat E,d\in\cat D$. We let $e\boxtimes d$ denote the image of $(e,d)$ under the canonical functor $\cat E\times \cat D\to\cat E\otimes\cat D$. 
\end{nota}
\begin{lm}
There is an equivalence $$\Map(x,x)\simeq \THH(\cat C, \Map(x,-)\boxtimes x)$$ where we identify $\Fun^L(\Ind(\C),\Sp) \otimes\Ind(\C)$  and $\Fun^L(\Ind(\C),\Ind(\C))$.
\end{lm}
\begin{proof}
The evaluation map in the canonical duality data for $\Ind(\C)$ is precisely given by the evaluation $F\boxtimes y \mapsto F(y)$. 
\end{proof}
\begin{obs}
The equivalence $\Fun^L(\Ind(\C),\Sp)\otimes\Ind(\C)\to \Fun^L(\Ind(\C),\Ind(\C))$ is given by $F\boxtimes y \mapsto (z\mapsto F(z)\otimes y)$. 

In particular, using this identification, we find a natural map $\mathrm{can}: \Map(x,-)\boxtimes x\to \id$ given by the natural map $\Map(x,z)\otimes x\to z$ adjoint to the identity $\Map(x,z)\to \Map(x,z)$. 
\end{obs}
\begin{cons}
Identifying $\Map(x,x)$ as $\THH(\C, \Map(x,-)\boxtimes x)$, and using the natural map $\mathrm{can}: \Map(x,-)\boxtimes x\to \id_\C$, this gives us a natural definition of $$\iota: \Map(x,x)\to \THH(\C)$$ as the trace of $\mathrm{can}$, which is the one we will be using.
\end{cons}
\begin{rmk}
The above map \emph{is} natural in $(\C,x)$, but we do not even need a full homotopy-coherent naturality statement to make the argument work. Therefore we will leave the details of the construction to the reader, as they are not needed and only a homotopically naive naturality is required. 
\end{rmk}

Furthermore, note that the dual of $\Ind(\C)$ in $\PrL_{st}$ is $\Ind(\C\op)$, but duality induces an equivalence $\C\simeq \C\op$. The main point of the construction is to follow along this identification and see what the trace becomes. 

Indeed, the evaluation morphism $\Ind(\C)\otimes\Ind(\C\op)\to \Sp$ is given by extending by filtered colimits the mapping spectrum functor $\Map : \C\otimes\C\op \to\Sp$. 

Letting $D$ denote the duality functor on $\C$, the functor $(x,y)\mapsto \Map(y,x)$ is equivalent to $(x,y)\mapsto \Map(\un,x\otimes Dy)$, and so the trace map $\Ind(\C)\otimes\Ind(\C)\to\Sp$ is identified with the following composite : $$\Ind(\C)\otimes\Ind(\C)\overset{\mu}\to \Ind(\C)\overset{\Map(\un,-)}\to \Sp$$ where $\mu$ is the tensor product of the induced monoidal structure on $\Ind(\C)$. 

In fact, we have the following more precise claim: 
\begin{lm}
Let $\mu^*$ denote the right adjoint to $\mu$, and $\eta: \Sp\to \Ind(\C)$ the unit of the symmetric monoidal structure on $\Ind(\C)$. The following two maps define a duality data on $\Ind(\C)$: 
$$\Map(\un,-)\circ \mu : \Ind(\C)\otimes \Ind(\C)\to \Sp$$ and $$\mu^*\circ \eta : \Sp \to \Ind(\C)\otimes\Ind(\C)$$
\end{lm}
\begin{proof}
By \cite[Proposition 2.20]{HSSS}, $\Ind(\C)$ is rigid, so the result is given by \cite[Proposition 2.17]{HSSS}. 

Alternatively, the identification above is a sketch of proof. 
\end{proof}
\begin{cons}
Letting $\tilde\id$ denote the image of $\id: \Ind(\C)\to \Ind(\C)$ under the identification $$\Fun^L(\Ind(\C),\Ind(\C))\simeq  \Ind(\C)\otimes\Ind(\C) $$ we find that to describe a morphism $\THH(\C)\to \Endo(\un)$, it suffices to describe a morphism $\mu(\tilde \id)\to \un$, to which we can then apply $\Map(\un,-)$. 

Describing such a morphism is equivalent to describing a morphism $\tilde \id \to \mu^*\un$. 

For this, note that the above lemma identifices $\mu^*\circ \eta : \Sp \to \Ind(\C)\otimes\Ind(\C)$ with the unique colimit preserving functor $\Sp\to \Fun^L(\Ind(\C),\Ind(\C))$ sending the sphere to $\id$, so this provides an identification $\tilde \id\simeq \mu^*(\un)$ (this identification can also be obtained by describing explicitly $\mu^*$ in terms of a restriction, as $\mu$ is left Kan extended from $\C\otimes\C$).

We therefore get a map $\mu(\tilde\id)\to \un$ adjoint to the above identification. 

This in turn induces a morphism $$\THH(\C)\to \Endo(\un)$$ 
\end{cons}
The claim is now that the composite $\Map(x,x)\to \Endo(\un)$ acts as desired, at least on $\pi_0$. 

\begin{rmk}\label{rmk : S1comm}
It is not too hard, if tedious, to prove that the map $\THH(\C)\to\Endo(\un)$ that we just described has another description, namely, we can view the map $\mu: \C\otimes\C\to \C$ as a map of commutative $\C$-algebras, and apply the $\C$-linear version of $\THH$ to it. 

For formal reasons this is well-defined and induces a mophism of $S^1$-equivariant commutative algebras in $\Ind(\C)$, $\eta (\THH(\C))\to \un $, and thus similarly a morphism of $S^1$-equivariant commutative ring spectra $\THH(\C)\to \Endo(\un)$, thus completely encoding the cyclic invariance of the trace. 

The proof that this is the same morphism as the one we described above is a tedious analysis of duality data, and of the functoriality of $\THH$. We do not need this statement in this paper so we do not adress this here. It can be done replacing $\PrL_{st}$ by any symmetric monoidal $(\infty,2)$-category. 
\end{rmk}
\begin{rmk}\label{rmk : agreemaps}
Consider the case where $\C = \Perf(R)$ for some commutative ring spectrum $R$. Then $\mu^*(\un)$ is the $R\otimes R$-module $R$, obtained by restricting scalars along $R\otimes R\to R$. This is indeed the $R$-bimodule corresponding to the identity, and $\mu$ is extension of scalars along the same map.

It follows that the co-unit $\mu\mu^*\un \to \un$ is the canonical map $R\otimes_{R\otimes R}R\to R$, and because $\Map(\un,-)$ is just the forgetful functor in this context, we find that this map $\THH(\C)\to \Endo(\un)$ coincides, in this special case, with the usual map $\THH(R)\to R$, as claimed in remark \ref{rmk : agree}.
\end{rmk}

Recall that our map $\iota: \Map(x,x)\to \THH(\C)$ is defined by applying $\THH(\C,-)$ to $$\mathrm{can}: \Map(x,-)\boxtimes x\to \id$$ 
Now these objects live in $\Fun^L(\Ind(\C),\Sp)\otimes\Ind(\C)\simeq \Fun^L(\Ind(\C),\Ind(\C))$.

\begin{lm}
Evaluation at $x$ and at $\id_x$ induces an equivalence $$\map_{\Fun^L(\Ind(\C),\Ind(\C))}(\Map(x,-)\otimes x, \id)\to \map(x,x)$$
\end{lm}
\begin{proof}
Note that $-\otimes x : \Sp\to \Ind(\C)$ is left adjoint to $\Map(x,-):\Ind(\C)\to \Sp$, so we find $$\map_{\Fun^L(\Ind(\C),\Ind(\C))}(\Map(x,-)\otimes x, -)\simeq \map_{\Fun^L(\Ind(\C),\Sp)}(\Map(x,-), \Map(x,-))$$

The spectral Yoneda lemma then implies that evaluation at $x$ and at $\id_x$ induces an equivalence of the latter with $\map(x,x)$.
\end{proof}

The natural map $\mathrm{can}: \Map(x,-)\boxtimes x \to \id$ corresponds to $\id_x$ under this equivalence, by definition. 

Now consider the evaluation map $Dx\otimes x\to \un$. It is of the form $\mu(Dx\boxtimes x)\to \un$, so it is adjoint to a certain map $Dx\boxtimes x\to \mu^*(\un)\simeq \tilde \id$; and under the identification $\Ind(\C)\otimes\Ind(\C)\simeq \Fun^L(\Ind(\C),\Ind(\C))$, this corresponds to a certain map $\Map(x,-)\boxtimes x\to \id$. 
\begin{prop}
The evaluation $Dx\otimes x\to \un$ induces $\mathrm{can}: \Map(x,-)\boxtimes x\to \id$
\end{prop}
\begin{proof}
By the previous lemma, it suffices to evaluate the map $\Map(x,-)\boxtimes x\to \id$ induced by the evaluation at $x$ and at $\id_x$, and check that we get something equivalent to $\id_x : x\to x$.

To do that properly, one needs to understand $\mu$ as a restriction. Namely, $\mu$ is a functor $\Ind(\C)\otimes \Ind(\C) \to \Ind(\C)$, and by viewing $\Ind(\C)\simeq \Fun^{ex}(\C\op,\Sp)$, $\mu$ is given by left Kan extension along the tensor product of $\C$: $\C\op\otimes\C\op\to\C\op$, so that $\mu^*$ is given by restriction along the same map. 

In particular, in $\Fun^{ex}(\C\op\otimes\C\op,\Sp)$, the map $Dx\boxtimes x\to \mu^*(\un)$ adjoint to evaluation  is the composite $$\Map(y,Dx)\otimes \Map(z,x)\to \Map(y\otimes z, Dx\otimes x)\to \Map(y\otimes z, \un)\simeq \Map(z,Dy)$$

The comparison to $\Map(x,-)\boxtimes x$ uses the equivalence $\Map(y,Dx)\simeq \Map( x, Dy)$. 

So evaluating the above composite at $Dy \simeq x$ shows that the map is $$\Map(x,x)\otimes \Map(z,x)\to \Map(Dx,Dx)\otimes \Map(z,x)\to \Map(Dx\otimes z, Dx\otimes x) \to Map(Dx\otimes z, \un) \simeq \Map(z,x)$$
naturally in $z$. 

It now suffices to evaluate at $\id_x$ by precomposing with $\Map(z,x)\to \Map(x,x)\otimes \Map(z,x)$. 

By the spectral Yoneda lemma, it again suffices to evaluate at $z=x$ and at $\id_x$ to prove that this is the identity $\Map(-,x)\to \Map(-,x)$, i.e. the identity $x\to x$. But this is now an easy check at a $\pi_0$-level.  

\end{proof}

The following corollary is then a form of theorem \ref{thm : main}.  
\begin{cor}
The composite $\Map(x,x)\to \THH(\C)\to \Endo(\un)$ is given on $\pi_0$ by taking traces.
\end{cor}
\begin{proof}
By the previous proposition, our map $\Map(x,x)\to \THH(\C)$ is exactly given by applying $\Map(\un, -)\circ \mu$ to the map $$Dx\boxtimes x\to \mu^*\un$$
adjoint to the evaluation $Dx\otimes x\to \un$. 

In other words, the composite $$Dx\otimes x= \mu(Dx\boxtimes x)\to \mu\mu^*\un \to \un$$ is exactly the evaluation map. 

Applying $\Map(\un,-)$ to the above composite yields the composite from the statement. But the above composite is the evaluation map, so applying $\Map(\un,-)$ to it yields the trace map on $\pi_0$, which is what we claimed. 
\end{proof}
We can now complete the proof of the main theorem: 
\begin{proof}[Proof of the main theorem \ref{thm : truemain}]
Let $\C$ be a stably symmetric monoidal $\infty$-category with unit $\un$, and let 
$$\xymatrix{X\ar[r] \ar[d]^f & Y \ar[r]\ar[d]^g & Z\ar[d]^h \\ 
X\ar[r] & Y\ar[r] & Z}$$ 
denote a co/fiber sequence of endomorphisms of dualizable objects.

The full subcategory $\C^{\mathrm{dbl}}$ of dualizable objects of $\C$ is a full stable subcategory, and so we may assume without loss of generality that all objects of $\C$ are rigid. 

We let $[f]$ denote the image of $f \in \Map(x,x)$ under the canonical map $\Map(x,x)\to \THH(\C)$. By theorem \ref{thm : main}, it suffices to prove that $[f] + [h] = [g]$ in $\pi_0\THH(\C)$.

This follows from what we explained in section \ref{section : strat} : $\THH$ is an additive invariant, and endomorphisms are represented by $\Perf(\Sph[t])$. More precisely, the co/fiber sequence of endomorphisms is represented by a co/fiber sequence of functors $\Perf(\Sph[t]) \to\C$, call them $F,G,H$, which, by additivity of $\THH$, induces an equivalence $\THH(F)+\THH(H) \simeq \THH(G)$. 

Evaluating this equivalence on the image of $t\in \Map(\Sph[t],\Sph[t])$ under the canonical map $\Map(\Sph[t],\Sph[t])\to \THH(\Perf(\Sph[t]))$ yields $[f]+[h]=[g]$, as desired. 

\end{proof}
\begin{rmk}
By \cite{KEnd}, there is a natural transformation of additive invariants from the $K$-theory of endomorphisms to $\THH$,  $\mathrm{KEnd}\to \THH$. 

There are also canonical maps $\map(x,x)\to \mathrm{KEnd}(\C)$, and $\mathrm{KEnd}$ is, essentially by definition, the universal place where the additivity theorem is satisfied - that is, where the images of $f,g,h$ satisfy $\overline g= \overline f+\overline h$. 

But this also means that defining a map $\mathrm{KEnd}(\C)\to \Endo(\un)$ without going through $\THH$ is ``philosophically'' unlikely: it would most likely require an argument such as the additivity theorem \emph{a priori}.

A contrario, the definition of $\THH$ in terms of a cyclic bar construction seems to suggest that all that is needed to define a trace map $\THH(\C)\to \Endo(\un)$ should be the cyclic invariance of the trace: no a priori knowledge of the additivity theorem should come in to that proof. This suggests that such a proof should be possible, and indeed it is. 
\end{rmk}
\begin{rmk}\label{rmk : noncom}
As pointed out in \cite[Remark 6.7]{HSS}, the morphism $\THH(\C)\to \Endo(\un)$ is not an instance of some general nonsense about additive invariants regarding $\THH$ : indeed it is only defined when $\C$ is \emph{rigid}, and in fact target is not even defined for general stable categories. 

Overall the proof could be separated in two parts: a noncommutative part consisting in describing a morphism $\mathrm{KEnd}\to \THH$ (which shows that the result holds ``in $\THH$''), and a commutative part, consisting in describing a morphism $\THH(\C)\to \Endo(\un)$.
\end{rmk}
\begin{rmk}
The case where $f,g,h$ are identities, so when one is computing ``Euler characteristics'', is already dealt with in \cite{HSS}. Going from $K$-theory to $\THH$, and constructing the morphism described above is what allows one to go from this special case to the full additivity theorem.
\end{rmk}
\begin{rmk}
The object $\mu(\tilde\id)= \mu\mu^*\un\in \Ind(\C)$ looks like an interesting object of $\Ind(\C)$, it is a $\C$-module version of $\THH(\C)$. Note that when $\C= \Perf(R)$ for some commutative ring spectrum $R$, it is simply $\THH(R)$ with its usual $R$-module structure. 

In particular, the $S^1$-action on $\THH(\C)$ generally does not lift to $\mu\mu^*\un$ (for instance, the $S^1$-action on $\THH(\mathbb F_p)$ is known not to be $\mathbb F_p$-linear). 
\end{rmk}
\section{A more general construction}\label{section : rig}
This is not useful to prove the additivity theorem, but the existence of a morphism from the Hochschild homology to the endomorphisms of the unit is interesting in its own right - we describe here a construction that works more generally. 

Suppose $\cat E\in\CAlg(\PrL_{st})$, and let $\PrL_\cat E$ denote its category of modules in $\PrL_{st}$. 

Suppose $\cat C\in \CAlg(\PrL_\cat E)$ is rigid (in the sense of \cite[Definition 2.15]{HSSS}). 

Then, by \cite[2.17]{HSSS}, it is dualizable in $\PrL_\cat E$ and its $\cat E$-linear trace is the following composite: $$\cat E\overset{\eta}\to \cat C\overset{\mu^*}\to \cat C\otimes_\cat E\cat C\overset{\mu}\to \cat C\overset{\eta^*}\to \cat E$$
where $\eta: \cat E\to \cat C$ is the unit, $\mu$ the multiplication, and $f^*$ is the ($\cat E$-linear) right adjoint of $f$. 

In particular, the ($\cat E$-linear) co-unit $\mu\mu^* \to \id_\cat C$ induces a map $\HH_\cat E(\cat C)\to \eta^*\eta\un_\cat E$, and the latter can be called $\Endo_\cat E(\un_\cat C)$. 

This is of course the same map as before when $\cat E= \Sp$, and when $\cat C= \Ind(\cat C^\omega)$, with $\cat C^\omega$ - only, this description is slightly less convenient to identify the composite $\map(x,x)\to \THH(\cat C)\to \Endo(\un)$. 

Even in the case where $\cat E= \Sp$, this is slightly more general than before, as it does not require $\cat C$ to be compactly generated - see \cite{KDual} for an account of additive or localizing invariants applied to dualizable presentable stable categories, following Efimov. See specifically Example 11 therein, dealing with $\THH$. 
\begin{rmk}
Remark \ref{rmk : S1comm} applies here as well. The rigidity of $\C$ guarantees that the multiplication map $\mu : \C\otimes_\cat E\C\to \C$ is an internal left adjoint in $\C$-modules, and one can thus apply $\HH_\C$ to it. This shows that $\HH_\cat E(\C)\to \eta^*\eta\un_\cat E$ can be viewed as a morphism of  $S^1$-equivariant commutative algebras in $\cat E$. 

The verification that this morphism agrees with the previous one is the same as in the previous remark, and is also not adressed here. 
\end{rmk}
\section{Traces of finite colimits}\label{section : mobius}
In this section we explain how to get a more general formula for traces of finite colimits, based on an argument of Berman \cite{berman}. 

For this, the ``commutative part'' of the argument remains unchanged, while we need to adapt the ``noncommutative'' part (see remark \ref{rmk : noncom}) slightly. 
\newcommand{\A}{\mathbf A}
\newcommand{\B}{\mathbf B}

The idea is the following : based on Berman's work, we find a formula expressing $E(\mathrm{colim}_I f(i))$ in terms of the $E(f(i))$'s, for $E$ an additive invariant and $f:I\to \Fun^{ex}(\A,\B)$ a functor from a finite category $I$, where $\A,\B$ are stable. Then, using the fact that $\THH$ is such an invariant, we conclude in the same way as before.
\begin{rmk}
Recall that in our convention,  ``finite category'' means  ``finite $\infty$-category''. Note that a $1$-category which is finite as a $1$-category need \emph{not} be finite as an $\infty$-category. 
\end{rmk}

The main result from \cite{berman} is the following (note \cite[Remark 10]{berman}): 
\begin{thm}\label{thm : mainberman}
Let $\C$ be a finitely cocomplete category and $A$ an abelian group. Suppose $\chi : \pi_0(\C^\simeq)\to A$ is a function satisfying the following : 
\begin{enumerate}
    \item If $0$ is the initial object of $\C$, then $\chi(0) = 0$
    \item If $\xymatrix{A\ar[r] \ar[d] & B\ar[d] \\ C\ar[r] & D}$ is a pushout square, then $\chi(D)+\chi(A) = \chi(C)+\chi(B)$
\end{enumerate}
Then for any functor $f: I\to \C$ from a finite category $I$, we have $$\chi(\mathrm{colim}_If(i)) = \sum_{i\in \pi_0(I^\simeq)} \mu_I(i)\chi(f(i))$$.
\end{thm}
\begin{rmk}
$\mu_I$ is a \emph{Möbius function} for $I$, \cite[Theorem 3]{berman} gives an explicit formula for it when $I$ is presented by a finite simplicial set. 
\end{rmk}
We apply this theorem for the map $\pi_0(\C^\simeq)\to K_0(\C)$ which obviously satisfies the conditions of the theorem when $\C$ is stable, so we get :
\begin{cor}
Let $\C$ be a stable category. Then, for any finite category $I$ and functor $f: I\to \C$, we have the following equality in $K_0(\C)$ (where $[x]$ denotes the class in $K_0(\C)$ of the object $x$): $$[\mathrm{colim}_If(i)] = \sum_{i\in \pi_0(I^\simeq)}\mu_I(i) [f(i)]$$
\end{cor}
\begin{cor}\label{cor : univcolim}
Suppose $\A$ is a compact stable category, and $\B$ an arbitrary small stable category, and let $f: I\to \Fun^{ex}(\A,\B)$ be a functor from a finite category $I$. 

Let $\mathcal U$ denote the universal additive invariant \cite{BGT}. Then we have $\mathcal U(\mathrm{colim}_if(i)) \simeq \sum_{i\in\pi_0(I^\simeq)}\mu_I(i)\mathcal U(f(i))$ as morphisms $\mathcal U(\A)\to \mathcal U(\B)$. 
\end{cor}
\begin{proof}
Applying the previous corollary to $\C = \Fun^{ex}(\A,\B)$, we get the equality in $K_0(\Fun^{ex}(\A,\B))\cong \pi_0\map(\mathcal U(\A),\mathcal U(\B))$, where the last isomorphism is $\pi_0$ applied to the equivalence in \cite[Theorem 7.13]{BGT}. 
\end{proof}
\begin{rmk}
In principle, the cited reference only applies to idempotent-complete stable categories. This issue is irrelevant as all the proofs and statements for connective $K$-theory work equally well in the non-idempotent complete setting. 

In any case, the statements we want to get in the end only depend on our stable category up to idempotent completion, so the reader who wishes to do so may add the adjective ``idempotent-complete'' to the above statement and proceed as though we idempotent complete everything on the way. 
\end{rmk}
\begin{rmk}
By universality of $\mathcal U$, the same property holds for all additive invariants.
\end{rmk}
Note that $\Perf(\Sph[t])$ is a compact stable category by the equivalence \ref{eq : repend}, and so we may apply this corollary to $\Perf(\Sph[t])$. We then get the desired result for traces: 
\begin{cor}
Let $\C$ be a stably symmetric monoidal category, with unit $\un$. 

Let $I$ be a finite category, and $\mu_I$ a Möbius function for it. Suppose $X : I^\triangleright\to \C$ is a colimit diagram taking values in dualizable objects, and let $f: X\to X$ be an endomorphism of $f$. 

Then $$\tr(f \mid \mathrm{colim}_I X_i) = \sum_{i\in \pi_0(I^\simeq)} \mu_I(i)\tr(f \mid X_i)$$
\end{cor}
\begin{proof}
Up to corestricting $X$, we may assume $\C$ is small and every object therein is dualizable. 

In particular we get a morphism $\THH(\C)\to \Endo(\un)$ realizing the trace on $\pi_0$. 

Furthermore we can see $f$ as a colimit diagram $I^\triangleright \to \Fun^{ex}(\Perf(\Sph[t]),\C)$ and so, as morphisms $\mathcal U(\Perf(\Sph[t]))\to \mathcal U(\C)$ we have the equality from the previous corollary. $\THH$ being an additive invariant, we have the same equality as morphisms $\THH(\Perf(\Sph[t]))\to \THH(\C)$, which is enough to conclude as in the first section. 
\end{proof}
\begin{rmk}
There is an arguably easier proof, applying theorem \ref{thm : mainberman} to the trace function on $\Fun(\Delta^1/\partial \Delta^1,\C)$ with values in $A=\pi_0\Endo(\un)$. 

The end result is essentially the same, but we thought it worthwhile to formulate corollary \ref{cor : univcolim} even if a proof can be given without it. 
\end{rmk}
Specializing to $f= \id_X$ we get :
\begin{cor}
Let $\C$ be a stably symmetric monoidal category, with unit $\un$. 

Let $I$ be a finite category, and $\mu_I$ a Möbius function for it. Suppose $X : I^\triangleright\to \C$ is a colimit diagram taking values in dualizable objects. Let $\chi$ denote the Euler characteristic of dualizable objects. 

Then  $$\chi( \mathrm{colim}_I X_i) = \sum_{i\in \pi_0(I^\simeq)} \mu_I(i)\chi( X_i)$$
\end{cor}
\bibliographystyle{alpha}
\bibliography{Biblio.bib}

\end{document}